\documentclass{amsart}
\usepackage{amsmath}
\usepackage{amssymb}
\usepackage{color}

\input xy
\xyoption{all}
\usepackage[all,cmtip]{xy}

\theoremstyle{definition}
\newtheorem{theorem}{Theorem}[section]
\newtheorem{lemma}[theorem]{Lemma}

\theoremstyle{definition}
\newtheorem{definition}[theorem]{Definition}

\newtheorem{c-example}[theorem]{Counter-example}
\newtheorem{Lemma}[theorem]{Lemma}
\newtheorem{corollary}[theorem]{Corollary}
\newtheorem{Prop}[theorem]{Proposition}
\theoremstyle{remark}
\newtheorem{remark}[theorem]{Remark}

\newtheorem*{theo-non}{\bf Theorem 2.3}
\numberwithin{equation}{section}

\usepackage{hyperref}
\usepackage{cancel}

\newcommand{\Cal}[1]{{\mathcal #1}}
\newcommand{\paral}[1]{\ar@<0.3ex>[#1] \ar@<-0.3ex>[#1]}

\newcommand{\PreOrd}[1]{\mathsf{PreOrd}(\mathbb{#1})}

\DeclareMathOperator{\Sub}{\mathsf{Sub}}

\newdir{ >}{{}*!/-7pt/@{>}}

\begin{document}

 \title{The stable category of preorders in a pretopos II: the universal property}
    \author{Francis Borceux}
\address{Universit\'e catholique de Louvain, Institut de Recherche en Math\'ematique et Physique, 1348 Louvain-la-Neuve, Belgium}
\email{francis.borceux@uclouvain.be}

\author{Federico Campanini}
\address{Universit\`a degli Studi di Padova, Dipartimento di Matematica ``Tullio Levi-Civita'', 35121 Padova, Italy}
\email{federico.campanini@unipd.it}

\author{Marino Gran}
\address{Universit\'e catholique de Louvain, Institut de Recherche en Math\'ematique et Physique, 1348 Louvain-la-Neuve, Belgium}
\email{marino.gran@uclouvain.be}

\thanks{
The second author was partially supported by Fondazione Ing. Aldo Gini - Universit\`a di Padova, borsa di studio per l'estero bando anno 2019 and by Ministero dell'Istruzione, dell'Universit\`a e della Ricerca (Progetto di ricerca di rilevante interesse nazionale ``Categories, Algebras: Ring-Theoretical and Homological Approaches (CARTHA)'').\\
This work was also supported by the collaboration project \emph{Fonds d'Appui à l'Internationalisation} ``Coimbra Group'' (2018-2021) funded by the Universit\'e catholique de Louvain.}



\begin{abstract} 
We prove that the \emph{stable category} associated with the category $\mathsf{PreOrd}(\mathbb C)$ of internal preorders in a pretopos $\mathbb C$ satisfies a universal property. The canonical functor from $\mathsf{PreOrd}(\mathbb C)$ to the stable category $\mathsf{Stab}(\mathbb C)$ universally transforms a pretorsion theory in $\mathsf{PreOrd}(\mathbb C)$ into a classical torsion theory in the pointed category $\mathsf{Stab}(\mathbb C)$. This also gives a categorical insight into the construction of the stable category first considered by Facchini and Finocchiaro in the  special case when $\mathbb C$ is the category of sets.
\end{abstract}

\maketitle
\section*{Introduction}
This article is meant as the sequel of \cite{BCG} and it deals with the study of the universal property of the $\emph{stable category}$ $\mathsf{Stab}(\mathbb C)$ of the category $\mathsf{PreOrd}(\mathbb C)$ of internal preorders  in a pretopos $\mathbb C$. It reveals the categorical feature of a natural construction 
due to A. Facchini and C. Finocchiaro in the category $\mathsf{PreOrd}$ of preordered sets  \cite{FF}, that we first briefly recall.

The category $\mathsf{PreOrd}$ contains the full subcategories $\mathsf{Eq}$ and $\mathsf{ParOrd}$ 
whose objects are equivalence relations and partially ordered sets, respectively. The pair of categories $(\mathsf{Eq}, \mathsf{ParOrd})$ has two properties making it a pretorsion theory in $\mathsf{PreOrd}$ \cite{FF, FFG}, that is, a kind of ``non-pointed  torsion theory''. More explicitly, any preordered set $(A, \rho)$, where $\rho$ is a reflexive and transitive relation on the set $A$, determines an equivalence relation $(A, \sim_{\rho})$, where $\sim_{\rho} = \rho \cap {\rho}^o$ and $\rho^o$ is the opposite relation of $\rho$, and a partially ordered set $({A/\sim_{\rho}}, \pi({\rho}))$, where $ \pi \colon A \rightarrow A/\sim_{\rho}$ is the quotient of $A$ by the equivalence relation $\sim_{\rho}$, and $\pi({\rho})$ is the partial order induced by $\rho$ on the quotient $A/\sim_{\rho}$. This yields a short $\mathcal Z$-exact sequence 
$$\xymatrix{
(A, \sim_{\rho}) \ar[r]^{\mathsf{Id}_A} & (A,\rho)\ar[r]^-\pi & ({A}/{\sim_{\rho}}, \pi(\rho))\\
} \qquad \quad (\mathsf{SES})$$
where $\mathcal Z$ is the full subcategory of $\mathsf{PreOrd}$ whose objects are the ``trivial preorders'' $(B, =)$, with $B$ a set and $=$ the equality relation on $B$. This subcategory $\mathcal Z$ determines an \emph{ideal} of trivial morphisms \cite{Ehr}, where a morphism is called \emph{trivial} if it factors through a trivial object. The fact that the above sequence is $\mathcal Z$-exact means that the identity morphism $\mathsf{Id}_A$ above is the $\mathcal Z$-kernel of $\pi$, and the quotient $\pi$ is the $\mathcal Z$-cokernel of $\mathsf{Id}_A$, where the notions of $\mathcal Z$-kernel and $\mathcal Z$-cokernel are defined by the same universal properties characterizing usual kernels and cokernels, with the only difference that the ideal of $0$-morphisms is replaced by the ideal of trivial morphisms \cite{FFG}. The $\mathcal Z$-exact sequence $(\mathsf{SES})$ has the property that the $\mathcal Z$-kernel belongs to $\mathsf{Eq}$ (the ``torsion subcategory'') and the $\mathcal Z$-cokernel belongs to $\mathsf{ParOrd}$ (the ``torsion-free subcategory''). Furthermore, one easily sees that any order preserving morphism from an equivalence relation to a partial order is trivial. These two properties express the fact that $(\mathsf{Eq}, \mathsf{ParOrd})$ is a pretorsion theory in $\mathsf{PreOrd}$.

In their study of this pretorsion theory in $\mathsf{PreOrd}$ the authors of \cite{FF} introduced a new category, the stable category $ \mathsf{Stab}$ of preordered sets: this is a pointed category, arising as a quotient category, with the property that the canonical functor from $\mathsf{PreOrd}$ to $\mathsf{Stab}$ sends the trivial objects in $\mathcal Z$ to the zero object in $\mathsf{Stab}$, and any trivial morphism in  $\mathsf{PreOrd}$ to a zero morphism in $\mathsf{Stab}$.

In the first article \cite{BCG} of this series we proved that, whenever $\mathbb C$ is a coherent category \cite{Elep}, it is possible to give a purely categorical construction of the stable category $\mathsf{Stab}(\mathbb C)$ of the category $\mathsf{PreOrd}(\mathbb C)$ of internal preorders in $\mathbb C$ (we recall this construction in the first section of this article). Moreover, when $\mathbb C$ is a pretopos, the functor $\Sigma \colon \mathsf{PreOrd}(\mathbb C) \rightarrow \mathsf{Stab}(\mathbb C)$ preserves coproducts and sends short $\mathcal Z$-exact sequences in $\mathsf{PreOrd}(\mathbb C)$ to short exact sequences in the pointed category $\mathsf{Stab}(\mathbb C)$ (Theorem $7.14$ in \cite{BCG}). 

The aim of this article is to prove the universal property of the stable category $\mathsf{Stab}(\mathbb C)$, that relies on these two properties of the functor $\Sigma \colon \mathsf{PreOrd}(\mathbb C) \rightarrow \mathsf{Stab}(\mathbb C)$. If we call a functor $G \colon \mathsf{PreOrd}(\mathbb C) \rightarrow \mathbb X$ a \emph{torsion theory functor} (Definition \ref{tt-functor}) when it sends the torsion and the torsion-free subcategories of the pretorsion theory $(\mathsf{Eq} (\mathbb C), \mathsf{ParOrd}(\mathbb C))$ into the torsion and the torsion-free subcategory, respectively, of a torsion theory $({\mathcal T}, {\mathcal F})$ in the category $\mathbb X$, the universal property can be expressed as follows:\\

\textbf{Theorem~\ref{universal}:} The canonical functor $\Sigma \colon \mathsf{PreOrd}(\mathbb C) \rightarrow \mathsf{Stab}(\mathbb C)$ is universal among all finite coproduct preserving torsion theory functors $G \colon \mathsf{PreOrd}(\mathbb C) \rightarrow \mathbb{X}$, where $\mathsf{PreOrd}(\mathbb C)$ is equipped with the pretorsion theory $(\mathsf{Eq}(\mathbb C), \mathsf{ParOrd}(\mathbb C))$, and
$\mathbb X$ is a pointed category with coproducts equipped with a torsion theory $({\mathcal T}, {\mathcal F})$. This means that any finite coproduct preserving torsion theory functor $G \colon \mathsf{PreOrd}(\mathbb C) \rightarrow \mathbb X$ factors uniquely through $\Sigma$:
$$
\xymatrix{\mathsf{PreOrd}(\mathbb C)  \ar[rr]^{\Sigma} \ar[rd]_{\forall G} & & \mathsf{Stab}(\mathbb C) \ar@{.>}[ld]^{\exists ! \overline{G}} \\
& \mathbb X, &
}
$$
i.e. there is a unique functor $\overline{G}$ such that $\overline{G} \cdot  \Sigma = G$.
The induced functor $\overline{G}$ preserves finite coproducts, and it is a torsion theory functor.\\

This theorem reveals the nature of the stable category, namely to transform a pretorsion theory in the sense of \cite{FFG} into a ``classical'' torsion theory, universally. Note that some further properties of the stable category $\mathsf{Stab}(\mathbb C)$ can be established when the base category $\mathbb C$ is what we called a \emph{$\tau$-pretopos} in \cite{BCG}, that is a pretopos with the additional property that the transitive closure of any relation on an object exists. Under this assumption it is possible to show that, for any ``suitable'' category $\mathbb X$, the induced torsion theory functor $\overline{G}$ above preserves kernels and cokernels, hence in particular short exact sequences (Theorem \ref{universal-tau}).

\section{Preliminaries}
In this work we shall be mainly interested in the category $\mathsf{PreOrd}(\mathbb C)$ of internal preorders in a pretopos $\mathbb C$. In \cite{FFG2} it was proven that the there is a pretorsion theory $(\mathsf{Eq}(\mathbb C), \mathsf{Par}(\mathbb C))$ in $\mathsf{PreOrd}(\mathbb C)$. In order to make the present paper as self-contained as possible, we now recall all the definitions needed in the sequel.

\subsection*{Pretorsion theories}
We first briefly recall the definition of pretorsion theory for general categories, as defined in \cite{FF, FFG}. Let $\Cal C$ be an arbitrary category and consider a pair $(\Cal T, \Cal F)$ of two replete full subcategories of $\Cal C$. Set also $\Cal Z:=\Cal T\cap\Cal F$ and call it \emph{the class of ``trivial objects"}. A morphism $f\colon A\to A'$ in $\Cal C$ is $\Cal Z$-trivial if it factors through an object of $\Cal Z$. Notice that the class of trivial morphisms in $\Cal C$ is an ideal of morphisms in the sense of Ehresmann \cite{Ehr}, and thus it is possible to consider the notions of $\mathcal Z$-kernel and of $\mathcal Z$-cokernel, defined by replacing, in the definition of kernel and cokernel, the ideal of zero morphisms with the ideal of trivial morphisms induced by the subcategory $\mathcal Z$. More precisely, we say that a morphism $\varepsilon\colon X\to A$ in $\Cal C $ is a \emph{$\Cal Z$-kernel} of $f\colon A \to A'$ if $f\varepsilon$ is a $\Cal Z$-trivial morphism and whenever $\lambda \colon Y\to A$ is a morphism in $\Cal C$ and $f\lambda$ is $\Cal Z$-trivial, then there exists a unique morphism $\lambda'\colon Y\to X$ in $\Cal C$ such that $\lambda=\varepsilon\lambda'$.
The notion of \emph{$\Cal Z$-cokernel} is defined dually. A sequence $A\overset{f}{\to}B\overset{g}{\to}C$ is a \emph{short $\Cal Z$-exact sequence} if $f$ is the $\Cal Z$-kernel of $g$ and $g$ is the $\Cal Z$-cokernel of $f$. We say that the pair $(\Cal T,\Cal F)$ is a \emph{pretorsion theory} in $\Cal C$ if the following two properties are satisfied:
\begin{itemize}
\item
any morphism from an object $T\in\Cal T$ to an object $F\in\Cal F$ is $\Cal Z$-trivial;
\item
for every object $X$ of $\Cal C$ there is a short $\Cal Z$-exact sequence
$$\xymatrix{ T_X \ar[r]^f &  X \ar[r]^g &  F_X}$$ with $T_X\in\Cal T$ and $F_X\in\Cal F$.
\end{itemize}

\subsection*{Pretoposes}
In this article $\mathbb C$ will always be assumed to be a \emph{pretopos} (see \cite{Elep} for more details). Let us recall that $\mathbb C$ is a pretopos when 
\begin{itemize}
\item $\mathbb C$ is \emph{exact} (in the sense of Barr \cite{Barr}), \item $\mathbb C$ has finite sums (= coproducts), \item $\mathbb C$ is \emph{extensive} \cite{CLW}. \end{itemize}
The property of extensivity means that
$\mathbb C$ has pullbacks along coprojections in a sum and the following condition holds: in any commutative diagram, where the bottom row is the sum of $A$ and $B$
$$
\xymatrix{
A' \ar[r] \ar[d] & C \ar[d] & B' \ar[l] \ar[d]\\
A \ar[r]_-{s_1} & A \coprod B & B,\ar[l]^-{s_2}
}
$$
the top row is a sum if and only if the two squares are pullbacks.  
The property saying that the upper row of the diagram is a sum whenever the two squares are pullbacks is usually called the ``universality of sums". 

 Recall that a sum of two objects $A$ and $B$ is called \emph{disjoint} if the coprojections $s_1 \colon A\to A \coprod B$ and $s_2 \colon B\to A \coprod B$ are monomorphisms and their intersection $A\cap B$ in the pullback
 $$
 \xymatrix{A \cap B \ar[r] \ar[d] & B \ar[d]^-{s_2} \\
 A \ar[r]_-{s_1} & A \coprod B
 }
 $$
  is an initial object in the category $\Sub(A\coprod B)$ of subobjects of $A \coprod B$. 
For a finitely complete category $\mathbb C$ with finite sums, extensivity is equivalent to the property of having disjoint and universal finite sums. In a pretopos the supremum $A \cup B$ of two disjoint subobjects $A \rightarrow X$ and $B \rightarrow X$ is given by the coproduct $A \coprod B$ of these two objects in $\mathbb C$ (see Corollary $1.4.4$ in \cite{Elep}).
Recall also that any pretopos has a \emph{strict} initial object, namely an initial object $0$ with the property that any morphism with codomain $0$ is an isomorphism.

\subsection*{Internal preorders}
As already said, in this work we shall be mainly interested in the category $\mathsf{PreOrd}(\mathbb C)$ of internal preorders in a pretopos $\mathbb C$, that is defined as follows.

An object $(A, \rho)$ in $\mathsf{PreOrd}(\mathbb C)$ is a relation $\langle r_1,r_2 \rangle \colon \rho \rightarrow A \times A$ on $A$, i.e. a subobject of $A \times A$, that is \emph{reflexive}, i.e. it contains the ``discrete relation'' $\langle 1_A, 1_A \rangle \colon A \rightarrow A \times A$ on $A$ (also denoted by $\Delta_A$), 
and \emph{transitive}: there is a
morphism $\tau \colon \rho \times_A \rho \rightarrow \rho$ such that $r_1 \tau =  r_1 p_1$ and $r_2 \tau =  r_2 p_2$, where $(\rho \times_A \rho, p_1, p_2)$ is the pullback
$$
\xymatrix{ \rho \times_A \rho \ar[r]^-{p_2} \ar[d]_-{p_1}& \rho \ar[d]^{r_1} \\
\rho \ar[r]_{r_2} & A.
}
$$ 

A morphism $(A, \rho) \rightarrow  (B, \sigma)$ in the category $\mathsf{PreOrd}(\mathbb C)$ of preorders in $\mathbb C$ is a pair of morphisms $(f,\hat{f})$ in $\mathbb C$ making the following diagram commute
$$
\xymatrix{  \rho \ar@<.5ex>[d]^{r_2} \ar@<-.5ex>[d]_{r_1}  \ar[r]^{\hat{f}}  & \sigma \ar@<.5ex>[d]^{s_2} \ar@<-.5ex>[d]_{s_1}  \\
 A \ar[r]_{f} & {B,} &
}
$$
so that $f r_1= s_1 \hat{f} $ and $f r_2= s_2 \hat{f}$.

A preorder $(A, \rho)$ is called an \emph{equivalence relation} if there is a ``symmetry'', namely a morphism $s \colon \rho \rightarrow \rho$ such that $r_1 s = r_2$  and $r_2 s = r_1$.  Equivalently, the opposite relation $\rho^\circ$ of $\rho$ is isomorphic to $\rho$, hence they determine the same subobject of $A \times A$: $\rho^\circ = \rho$. 
A preorder $(A, \rho)$ is called a \emph{partial order} if it ``antisymmetric'', i.e. if it has the additional property that $\rho \cap \rho^\circ$ is equal to the discrete equivalence relation $\Delta_A$ on $A$. We write $\mathsf{Eq}(\mathbb C)$ and $\mathsf{Par}(\mathbb C)$ for the full (replete) subcategories of $\mathsf{PreOrd}(\mathbb C)$ whose objects are equivalence relations and partial orders in $\mathbb C$, respectively. The pair $(\mathsf{Eq}(\mathbb C), \mathsf{Par}(\mathbb C))$ is a \emph{pretorsion theory} in $\mathsf{PreOrd}(\mathbb C)$ \cite{FFG2}.
We write $\mathcal Z =\mathsf{Eq}(\mathbb C) \cap  \mathsf{Par}(\mathbb C)$ for the full (replete) subcategory of \emph{trivial objects} in $\mathsf{PreOrd}(\mathbb C)$ \cite{BCG}, whose objects are ``discrete'' preorders, i.e. those of the form $(A, \Delta_A)$.
A morphism $(f,\hat{f}) \colon (A, \rho) \rightarrow  (B, \sigma)$ is called a $\mathcal Z$-\emph{trivial morphism} if it factors through a \emph{trivial object}. In the following we shall often use the terms ``trivial morphism'' and ``trivial object'' (dropping the ``$\mathcal Z$'' of ``$\mathcal Z$-trivial'').

Given a morphism $f \colon A \rightarrow B$, where $(B, \sigma)$ is an object in $\mathsf{PreOrd}(\mathbb C)$, we denote by $f^{-1} (\sigma)$ the \emph{inverse image}
of $\sigma$ along $f$, that is the left vertical relation defined by the following pullback:
$$
\xymatrix{f^{-1} (\sigma) \ar[r] \ar[d] & \sigma \ar[d]^{\langle s_1, s_2\rangle} \\
A \times A   \ar[r]_{f \times f}& B \times B
}
$$
Recall then that, in any category with an initial object $0$, a subobject $\alpha \colon A \rightarrow B$  of an object $B$ is \emph{complemented} if there is another subobject $\alpha^c \colon A^c \rightarrow B$ with the property $A \cap A^c = 0$ and $A \cup A^c = B$.

It was observed in \cite{BCG} (Corollary $5.4$) that a subobject $\xymatrix{{(A, \rho)\, \, } \ar@{>->}[r]^\alpha &  (B, \sigma)}$ in $\mathsf{PreOrd}(\mathbb C)$
 is \emph{complemented} in $\mathsf{PreOrd}(\mathbb C)$ if and only if
 \begin{enumerate}
\item $\xymatrix{ {A\,\, } \ar@{>->}[r]^-{\alpha} & B}$ is a complemented subobject in $\mathbb C$, with complement $\xymatrix{ {{A}^c \,\, } \ar@{>->}[r]^-{\alpha^c} & B}$;
\item $\alpha^{-1} (\sigma) = \rho$, ${(\alpha^c)}^{-1} (\sigma) = {\rho}^c$ and 
all the commutative squares in the diagram
$$
\xymatrix@=40pt{
\rho \ar@{ >->}[r]^-{} \ar@<.5ex>[d]^{r_2} \ar@<-.5ex>[d]_{r_1} & \sigma  \ar@<.5ex>[d]^{s_2} \ar@<-.5ex>[d]_{s_1} & \rho^c \ar@{ >->}[l] \ar@<.5ex>[d]^{r_2^c} \ar@<-.5ex>[d]_{r_1^c}\\
A \ar@{ >->}[r]^-{\alpha} & B = A\coprod A^c & A^c \ar@{ >->}[l]_-{\alpha^c} \\
}
$$
\noindent (i.e. the ones corresponding to the same index $i \in \{1,2 \}$) are pullbacks.
\end{enumerate}
Note that, in a pretopos $\mathbb C$, this implies that $\sigma = \rho \coprod \rho^c$. 

\subsection*{Partial maps}
Before recalling the definition of the stable category of $\mathsf{PreOrd}(\mathbb C)$,
 as an intermediate step, we first define the category $\mathsf{PaPreOrd}(\mathbb C)$ of \emph{partial morphisms} in $\mathsf{PreOrd}(\mathbb C)$.
 
 Its objects are the same as the ones of $\mathsf{PreOrd}(\mathbb C)$, the internal preorders $(A, \rho)$ in $\mathbb C$, while a morphism $(A,\rho) \rightarrow (B,\sigma)$ in the category $\mathsf{PaPreOrd}(\mathbb C)$ is a pair $(\alpha, f)$ depicted as 
$$
\xymatrix{& (A',\rho') \ar@{ >->}[dl]_{\alpha} \ar[dr]^{f} & \\ (A, \rho) \ar@{.>}[rr]_{(\alpha,f)}& & (B,\sigma), }
$$
where $(A', \rho')$ is an internal preorder, $f$ is a morphism in $\mathsf{PreOrd}(\mathbb C)$, and $\alpha \colon (A',\rho') \rightarrow (A,\rho)$ is a complemented subobject in $\mathsf{PreOrd}(\mathbb C)$.
Given two composable morphisms $(\alpha, f) \colon (A,\rho) \rightarrow (B,\sigma)$ and $(\beta, g) \colon (B,\sigma) \rightarrow (C,\tau)$ in $\mathsf{PaPreOrd}(\mathbb C)$, the composite morphism $(\beta, g) \circ (\alpha, f)$ in $\mathsf{PaPreOrd}(\mathbb C)$ is defined by the external part of the following diagram
$$ \xymatrix@=15pt{
& & (A'', \rho'')  \ar[rd]^{f'} \ar@{ >->}[dl]_{\alpha'} & & \\
& (A',\rho') \ar@{ >->}[dl]_{\alpha} \ar[dr]^{f}& & (B',\sigma') \ar@{ >->}[dl]_{\beta} \ar[dr]^{g}& \\
(A,\rho) \ar@{.>}[rr]_{(\alpha,f)}& & (B,\sigma) \ar@{.>}[rr]_{(\beta,g)} & &  (C,\tau) }
$$
where the upper part is a pullback. In other words,
$$
(\beta, g) \circ (\alpha, f)= (\alpha \alpha', g f').
$$

The well-known properties of pullbacks guarantee that this composition is associative. For any preorder $(A,\rho)$,
the identity on it in $\mathsf{PaPreOrd}(\mathbb C)$ is the arrow
$$\xymatrix{& (A,\rho) \ar@{=}[dl]_{1} \ar@{=}[dr]^{1} & \\ (A, \rho) \ar@{.>}[rr]_1& & (A,\rho) }
 $$
 
\begin{remark}
 Notice that given a partial map $(\alpha, f)\colon (A, \rho)\to (B, \sigma)$, the subobject $\alpha \colon (A',\rho') \rightarrow (A,\rho)$ can only be determined up to isomorphism, being the representative of a class of monomorphisms $(A', \rho')\to (A, \rho)$ in $\PreOrd C$. Nevertheless, it is easy to prove that the composition is independent of the choice of representatives. See \cite{RR} and the references therein for more details about partial maps.
\end{remark}

 As explained in \cite{BCG}, there is a functor $I \colon \mathsf{PreOrd}(\mathbb C) \rightarrow \mathsf{PaPreOrd}(\mathbb C)$ which is the identity on objects and such that, for any $f \colon (A,\rho) \rightarrow (B,\sigma)$ in $\mathsf{PreOrd}(\mathbb C)$, its value $I(f) \colon (A,\rho) \rightarrow (B,\sigma)$ in $\mathsf{PaPreOrd}(\mathbb C)$ is given by the morphism
 $$
 \xymatrix{& (A,\rho) \ar@{=}[dl]_{1} \ar[dr]^{f} & \\ (A, \rho) \ar@{.>}[rr]_{I(f)}& & (B,\sigma).}$$
 
To simplify the notation, from now on, we shall write $A$ instead of $(A,\rho)$ to denote an internal preorder and $\xymatrix{A \ar[r] & B }$ for a morphism of preorders. 
 The fact that the initial object $0$ of $\PreOrd C$ is strict implies that $0$ is a zero object in $ \mathsf{PaPreOrd}(\mathbb C)$, and thus
 the category $ \mathsf{PaPreOrd}(\mathbb C)$ is equipped with an ideal $\mathcal N$ of (null) morphisms \cite{Ehr}, where $\mathcal N$ is the class of morphisms in $\mathsf{PaPreOrd}(\mathbb C)$ of the form 
 $$
 \xymatrix{& 0 \ar@{>->}[dl]_{} \ar[dr]^{} & \\ B \ar@{.>}[rr]_0 & & C. }
 $$
 
The \emph{stable category} \cite{BCG} is defined as a suitable quotient of the category $\mathsf{PaPreOrd}(\mathbb C)$. In the special case when $\mathbb C$ is the category of sets this construction reduces to the one of the stable category by Facchini and Finocchiaro in \cite{FF}. In order to define the stable category, the following notion is needed:
 \begin{definition}
 A \emph{congruence diagram} in $\mathsf{PreOrd}(\mathbb C)$ is a diagram of the form
\begin{equation}\label{C-diagram}  
\xymatrix@=35pt{
{A_0^{1}}^c \ar@{ >->}[rr]^{{\alpha_0^1}^c} & & A_1  \ar@{>->}[ld]^{\alpha_1} \ar[dr]^{f_1}  & & \\
 A_0 \ar@{ >->}[drr]_{\alpha_0^2} \ar@{ >->}[rru]^{\alpha_0^1 } \ar@{ >->}[r]_{\alpha_0} & A & & B \\
 {A_0^{2}}^c \ar@{ >->}[rr]_{{\alpha_0^2}^c}  & & A_2 \ar@{>->}[ul]_{\alpha_2} \ar[ru]_{f_2} & &
 }
\end{equation}
where:
 \begin{itemize}
 \item any arrow of the form $\xymatrix{ \ar@{>->}[r] &  }$ represents a complemented subobject in $\mathsf{PreOrd}(\mathbb C)$;
 \item the two triangles commute;
 \item $\xymatrix{{A_0^{i}}^c  \ar@{>->}[r]^{{\alpha_0^i}^c} & A_i }$ is the complement in $A_i$ of the subobject $\xymatrix{A_0 \ar@{>->}[r]^{{\alpha_0^i}} & A_i }$;
 \item  $f_1  \alpha_0^1 = f_2 \alpha_0^2$;
 \item each $f_i {\alpha_0^i}^c $ is a trivial morphism.
 \end{itemize}
 \end{definition}
 Two parallel morphisms $(\alpha_1, f_1)$ and $(\alpha_2, f_2)$ in $\mathsf{PaPreOrd}(\mathbb C)$, depicted as
$$
\xymatrix{
\ar@{}[drrrrrrr]|{\mbox{and}} & A_1 \ar@{>->}[dl]_{\alpha_1} \ar[dr]^{f_1} & & & & & A_2 \ar@{>->}[dl]_{\alpha_2} \ar[dr]^{f_2} & \\
 A \ar@{.>}[rr]_{} & & B & & & A \ar@{.>}[rr]_{}& & B,
 }
$$
 are \emph{equivalent} if there is a congruence diagram of the form \eqref{C-diagram} between them. In this case one writes $(\alpha_1, f_1) \sim (\alpha_2, f_2)$. As shown in \cite{BCG}, the relation $\sim$ is an equivalence relation which is also compatible with the composition in $\mathsf{PaPreOrd}(\mathbb C)$, and is then a \emph{congruence} (in the sense of \cite{MacLane}) on the category $\mathsf{PaPreOrd}(\mathbb C)$.

 \begin{definition}\cite{BCG}
The quotient category $\mathsf{Stab}(\mathbb C)$ of $\mathsf{PaPreOrd}(\mathbb C)$ by the congruence $\sim$ defined above is called the \emph{stable category}. If $\pi \colon \mathsf{PaPreOrd}(\mathbb C) \rightarrow \mathsf{Stab}(\mathbb C)$ is the quotient functor, we also have a functor $$\Sigma = \pi \circ I \colon \mathsf{PreOrd}(\mathbb C) \longrightarrow \mathsf{Stab}(\mathbb C)$$ obtained by precomposing $\pi$ with the functor $I \colon \mathsf{PreOrd}(\mathbb C) \rightarrow \mathsf{PaPreOrd}(\mathbb C)$.
 \end{definition}

\begin{remark}
The definition of the stable category $\mathsf{Stab}(\mathbb C)$ of $\mathsf{PreOrd}(\mathbb C)$ actually depends on the class $\Cal Z$ of trivial objects. Thus we should write ``the stable category of $\mathsf{PreOrd}(\mathbb C)$ with respect to $\Cal Z$" and call it the ``$\Cal Z$-stable category of $\mathsf{PreOrd}(\mathbb C)$". Nevertheless, we prefer to follow the notation adopted in \cite{FF} and refert to $\mathsf{Stab}(\mathbb C)$ as the stable category associated with $\mathsf{PreOrd}(\mathbb C)$.
It is also worth noting that the construction we provide is based on the properties of the class $\Cal Z$ (such as the fact that it contains both the initial and the terminal objects or that it is closed under coproducts) that may not hold for any pretorsion theory in $\mathsf{PreOrd}(\mathbb C)$. Thus, a priori, it is not possible to construct the stable category for any pretorsion theory in $\mathsf{PreOrd}(\mathbb C)$.
\end{remark}
 
The stable category is pointed, and the zero object $0$ of $\mathsf{Stab}(\mathbb C)$ is the image by the functor $\Sigma$ of the initial object in $\mathsf{PreOrd}(\mathbb C)$.
As shown in \cite{BCG}, an object $A$ in $\mathsf{PreOrd}(\mathbb C)$ is such that $\Sigma (A) = 0$ if and only if $A$ is a ``discrete'' object, that is an object $A$ equipped with the preorder given by the discrete equivalence relation $\Delta_A$ on $A$. Moreover, $f \colon A \rightarrow B$ in $\mathsf{PreOrd}(\mathbb C)$ is a trivial morphism if and only if $\Sigma(f) = 0$ in $\mathsf{Stab}(\mathbb C)$.
More generally, one has the following result, where we write $<\alpha, f>$ for the image of the morphism 
$$
\xymatrix{& (A',\rho') \ar@{ >->}[dl]_{\alpha} \ar[dr]^{f} & \\ (A, \rho) & & (B,\sigma), }
$$
by the functor $\pi$:

\begin{Lemma}\label{zero-morphisms}
For a morphism $\xymatrix@=30pt{A  \ar[r]^{< \alpha, f>} &  B }$ in $\mathsf{Stab}(\mathbb C)$ the following conditions are equivalent:
\begin{enumerate}
 \item $< \alpha, f> = 0$;
 \item $f$ is a trivial morphism in $\mathsf{PreOrd}(\mathbb C)$.
\end{enumerate}
\end{Lemma}
 
\begin{proof}
The assumption $< \alpha, f> = 0$ implies that there is a congruence diagram of the form 
$$\xymatrix@=35pt{
{A_0^{1}}^c \ar@{ >->}[rr]^{{\alpha_0^1}^c} & & A'  \ar@{>->}[ld]^{\alpha} \ar[dr]^{f}  & & \\
A_0 \ar@{ >->}[drr]_{\alpha_0^2} \ar@{ >->}[rru]^{\alpha_0^1 } \ar@{ >->}[r]_{\alpha_0} & A   & &B \\
{A_0^{2}}^c \ar@{ >->}[rr]_{{\alpha_0^2}^c}  & & 0 \ar@{ >->}[ul] \ar[ru] & &
}$$
hence $ A_0 = 0= {A_0^{2}}^c$ (since $0$ is a strict initial object). This implies that $ {A_0^{1}}^c = A'$, ${\alpha_0^1}^c = 1_{A'}$ and $f$ is trivial on $A'$.
 
 Conversely, when $f$ is a trivial morphism, it suffices to build the following congruence diagram
$$\xymatrix@=35pt{
A_1 \ar@{=}[rr] & & A_1  \ar@{>->}[ld]^{\alpha} \ar[dr]^{f}  & & \\
0 \ar@{=}[drr] \ar@{>->}[rru]^{ } \ar@{>->}[r] & A   & &B \\
0 \ar@{=}[rr]  & & 0 \ar@{>->}[ul] \ar[ru] & &
 }
$$
showing that $< \alpha, f> = 0$.
\end{proof}

 Note also that the ``intuition'' here should be that a diagram
 \begin{equation}\label{morphism-f}
 \xymatrix{& A' \ar@{>->}[dl]_{\alpha} \ar[dr]^{f} & \\ A & & B}
 \end{equation}
``represents'' a morphism $ < \alpha, f>$ whose restriction on the (complemented) subobject $(A',\rho')$ of $(A,\rho)$ is $f$, and that is ``trivial'' on the complement of $(A',\rho')$ in $(A,\rho)$, as explained in the following:

  \begin{Prop}\label{justification}
If $\xymatrix{A \ar[r]^{< \alpha, f>} & B}$ is a morphism in $\mathsf{Stab}(\mathbb C)$, then the following diagram is commutative in $\mathsf{Stab}(\mathbb C)$, where $\xymatrix{{{A'}^c \,\, }\ar@{>->}[r]^{\alpha^c} & A}$ is the complement of $A'$ in $A$ in $\mathsf{PreOrd}(\mathbb C)$:
\begin{equation}\label{commutative-just}
\xymatrix@=30pt{\Sigma(A') \ar[rrd]^{\Sigma(f)} \ar[rd]_{\Sigma(\alpha)}& & \\
& \Sigma(A) \ar[r]^(.4){< \alpha, f>} & \Sigma(B) \\
\Sigma({A'}^c) \ar[rru]_0 \ar[ru]^{\Sigma({\alpha^c)}} & & 
}
\end{equation}
\end{Prop}
 \begin{proof}
 In order to see that $< \alpha, f>  \Sigma(\alpha) = \Sigma(f)$ it
 suffices to consider the diagram 
 $$\xymatrix@=25pt{  & & A'  \ar@{=}[ld]^{} \ar@{=}[dr]^{}  & & & \\
&  A' \ar@{=}[dl]^{}  \ar@{>->}[dr]^{\alpha} &  & A' \ar@{>->}[dl]_{\alpha} \ar[dr]^f & & \\
 A'   \ar@{.>}[rr]^{\Sigma(\alpha)} & & A \ar@{.>}[rr]^{<\alpha, f>}  && B
 } $$
 where the upper quadrangle is a pullback.
 On the other hand, the assumption that $A' \cap {A'}^c= 0$ implies that $< \alpha, f>  \Sigma({\alpha^c})=0$
 $$\xymatrix@=25pt{  & & 0  \ar[ld]^{} \ar[dr]^{}  & & & \\
&  {A'}^c \ar@{=}[dl]^{}  \ar@{>->}[dr]^{\alpha^c} &  & A' \ar@{ >->}[dl]_{\alpha} \ar[dr]^f & & \\
 {A'}^c   \ar@{.>}[rr]^{\Sigma(\alpha^c)} & & A \ar@{.>}[rr]^{<\alpha, f>}  && B
 } $$
 since the composite $0 \rightarrow A' \rightarrow B$ is obviously trivial. One concludes by Lemma \ref{zero-morphisms}.
 \end{proof}
 
 The following result (Lemma 7.11 in \cite{BCG}) will also be useful:
 \begin{lemma}\label{existence-cokernel}\cite{BCG}
Let us consider a morphism $< \alpha, f>$ in $\mathsf{Stab}(\mathbb C)$ represented by 
\begin{equation}\label{morphism-stab}
\xymatrix{& {\, \, \, \, A'} \ar@{>->}[dl]_{\alpha} \ar[dr]^{f} & \\ A & & B}
\end{equation}
and assume that for any complemented subobject $\xymatrix{{B' \, \, }\ar@{>->}[r] & B}$ the induced morphism $f^{-1} (B') \rightarrow B'$ has a $\mathcal Z$-cokernel in $\mathsf{PreOrd}(\mathbb C)$. Then the cokernel of $<\alpha, f>$ exists in $\mathsf{Stab}(\mathbb C)$, and 
$$\mathsf{coker} (<\alpha, f>) = \Sigma ({\mathcal Z}\mbox{-}\mathsf{coker}(f)). $$
\end{lemma}

 \section{ The universal property of $\mathsf{Stab}(\mathbb C)$}
 \begin{definition}\label{tt-functor}
 Let $({\mathbb A}, {\mathcal T},  {\mathcal F})$ be a category $\mathbb A$ with a given pretorsion theory $({\mathcal T},  {\mathcal F})$ in $\mathbb A$. If $({\mathbb B}, {\mathcal T'},  {\mathcal F'})$ is a pointed category $\mathbb B$ with a given torsion theory $({\mathcal T'},  {\mathcal F'})$ in it, we say that a functor $G \colon \mathbb A \rightarrow \mathbb B$ is a \emph{torsion theory functor} if the following two properties are satisfied:
 \begin{enumerate}
 \item $G(A) \in {\mathcal T}'$ for any $A \in \mathcal T$, $G(B) \in {\mathcal F}'$ for any $B \in \mathcal F$;
 \item if $T(A) \rightarrow A \rightarrow F(A)$ is the canonical short $\mathcal Z$-exact sequence associated with $A$ in the pretorsion theory $({\mathcal T},  {\mathcal F})$, then $$0 \rightarrow G(T(A)) \rightarrow G(A) \rightarrow G(F(A)) \rightarrow 0$$ is a short exact sequence in $\mathbb B$.
 \end{enumerate}
 \end{definition}
 When $\mathbb C$ is a pretopos, we write $(\mathsf{Eq}(\mathbb C), \mathsf{ParOrd}(\mathbb C))$ for the pretorsion theory in $\mathsf{PreOrd}(\mathbb C)$ where $\mathsf{Eq}(\mathbb C)$ is the category of equivalence relations and $\mathsf{ParOrd}(\mathbb C)$ the category of partial orders in $\mathbb C$.
 \begin{Prop}\label{Sigma-is-tt}
 The functor $\Sigma \colon \mathsf{PreOrd}(\mathbb C) \rightarrow \mathsf{Stab}(\mathbb C)$ is a torsion theory functor that preserves finite coproducts and monomorphisms.
 \end{Prop}
 \begin{proof}
 The fact that $(\mathsf{Eq}(\mathbb C), \mathsf{ParOrd}(\mathbb C))$ is a pretorsion theory was observed in \cite{FFG2} (for any exact category $\mathbb C$), while the preservation of finite coproducts and monomorphisms by the functor $\Sigma$ was established in Proposition $6.2$ and Proposition $6.1$ in \cite{BCG},  respectively. It remains to prove that $(\mathsf{Eq}(\mathbb C), \mathsf{ParOrd}(\mathbb C))$ is a torsion theory in the pointed category $\mathsf{Stab}(\mathbb C)$. Consider any morphism $< \alpha , f > \colon (A, \rho) \rightarrow (B, \sigma)$, where $\rho$ is an equivalence relation on $A$ and $\sigma$ a partial order on $B$, depicted as
 $$\xymatrix{& (A',\rho') \ar@{>->}[dl]_{\alpha} \ar[dr]^{f'} & \\ (A, \rho) \ar@{.>}[rr]_f& & (B,\sigma) }
 $$
The fact that $\rho$ is an equivalence relation and $\alpha$ a complemented subobject in $\mathsf{PreOrd}(\mathbb C)$ implies that also $\rho' = \alpha^ {-1} (\rho)$ is an equivalence relation (on $A'$). It follows that $f\colon (A', \rho') \rightarrow (B, \sigma)$ is a trivial morphism in $\mathsf{PreOrd}(\mathbb C)$ (since $(\mathsf{Eq}(\mathbb C), \mathsf{ParOrd}(\mathbb C))$ is a pretorsion theory in $\mathsf{PreOrd}(\mathbb C))$, hence a zero morphism in $\mathsf{Stab}(\mathbb C)$. 

Next, let us prove that the canonical short $\mathcal Z$-exact sequence in $\mathsf{PreOrd}(\mathbb C)$
 $$
\xymatrix{
(A, \sim_{\rho}) \ar[r]^{i} & (A,\rho)\ar[r]^-\pi & ({A}/{\sim_{\rho}}, \pi(\rho))\\
}
$$
associated with any internal preorder $(A,\rho)$, where ${\sim_{\rho}} = \rho \cap \rho^o$ and $i$ is the canonical inclusion, becomes a short exact sequence in $\mathsf{Stab}(\mathbb C)$.

First, Proposition $7.1$ in \cite{BCG} implies that $\Sigma (i) \colon (A, \sim_{\rho}) \rightarrow (A,\rho)$ is the kernel of $\Sigma (\pi) \colon (A,\rho) \rightarrow ({A}/{\sim_{\rho}}, \pi(\rho))$. 
To see that $\Sigma (\pi)$ is the cokernel of $\Sigma(i)$ we shall use Lemma \ref{existence-cokernel}. To apply this result, observe that, for any complemented subobject $(A', \rho')$ of $(A, \rho)$, the upper horizontal morphism in the pullback
$$
\xymatrix{(A', \rho'')\ar@{>->}[r] \ar[d]^{i'} & (A, \rho \cap \rho^o) \ar[d]_{i} \\
(A', \rho')\ar@{>->}[r] & (A, \rho)
}
$$
 is again a complemented subobject. This implies that $\rho''$ is the restriction to $A'$ of the equivalence relation $\rho \cap \rho^o$ on $A$, i.e. the following square is a pullback in $\mathbb C$:
 $$
\xymatrix{\rho'' \ar[r] \ar[d] & \rho \cap \rho^o \ar[d] \\
A' \times A' \ar[r]_{} & A \times A
}
$$
 This implies that $\rho''$ is an equivalence relation, and then the $\mathcal Z$-cokernel of $i'$ exists (by Proposition $7.3$ in \cite{BCG}).  The result then follows from Lemma \ref{existence-cokernel}.

 \end{proof}

\begin{theorem}\label{universal}
Let $\mathbb C$ be a pretopos. The functor $\Sigma \colon \mathsf{PreOrd}(\mathbb C) \rightarrow \mathsf{Stab}(\mathbb C)$ has the following property: it is universal among all finite coproduct preserving \emph{torsion theory functors} $G \colon \mathsf{PreOrd}(\mathbb C) \rightarrow \mathbb X$, where $\mathbb X$ has a torsion theory $({\mathcal T}, {\mathcal F})$ and finite coproducts. This means that any finite coproduct preserving torsion theory functor $G \colon \mathsf{PreOrd}(\mathbb C) \rightarrow \mathbb X$ factors uniquely through $\Sigma$:
$$ \xymatrix{\mathsf{PreOrd}(\mathbb C)  \ar[rr]^{\Sigma} \ar[rd]_{\forall G} & & \mathsf{Stab}(\mathbb C) \ar@{.>}[ld]^{\exists ! \overline{G}} \\
& \mathbb X. &
}
$$
Moreover, the induced functor $\overline{G}$ preserves finite coproducts, and is a torsion theory functor.
\end{theorem}

\begin{proof}
Since $\mathsf{PreOrd}(\mathbb C)$ and $\mathsf{Stab}(\mathbb C)$ have the same objects it is clear that the definition of the functor $\overline{G}$ on the objects is ``forced'' by $G$: 
$\overline{G} (A) = G(A)$, for any object $A$ in $\mathsf{Stab}(\mathbb C)$. Let then $< \alpha, f > \colon A \rightarrow B$ be a morphism in $\mathsf{Stab}(\mathbb C)$ (as in \eqref{morphism-stab}), and recall that it is then $[f,0]$, the morphism induced by the universal property of the coproduct $A' \coprod {A'}^c = A$, since the diagram \eqref{commutative-just} in Proposition \ref{justification} commutes.
Again, the condition $\overline{G} \circ \Sigma= G$ and the fact that $\overline{G}$ has to preserve binary coproducts force the definition of the functor $\overline{G}$ on morphisms:
$$
\overline{G}(< \alpha, f > ) = [G(f), 0].
$$

The above arguments already prove the uniqueness of the functor $\overline{G}$ with the above properties.
We still need to check that $\overline{G}$ is well-defined on morphisms, i.e. if $< \alpha, f > = < \overline{\alpha}, \overline{f} >$ in $\mathsf{Stab}(\mathbb C)$, then $\overline{G} (< \alpha, f >) = \overline{G} (< \overline{\alpha}, \overline{f} >)$ or, equivalently,
$$
[G(f), 0] = [G(\overline{f}), 0].
$$
Now, the assumption $< \alpha, f > =< \overline{\alpha}, \overline{f} >$ gives a congruence diagram 

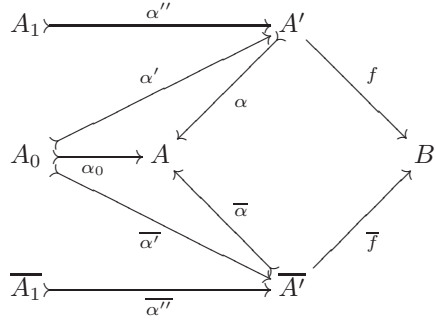
\begin{equation*}  
\xymatrix@=35pt{
A_1 \ar@{>->}[rr]^{\alpha''} & & A'  \ar@{>->}[ld]^{\alpha} \ar[dr]^{f}  & & \\
A_0 \ar@{ >->}[drr]_{\overline{\alpha'}} \ar@{ >->}[rru]^{\alpha'} \ar@{ >->}[r]_{\alpha_0} & A   & &B \\
\overline{A_1} \ar@{>->}[rr]_{\overline{\alpha''}}  & & \overline{A'} \ar@{>->}[ul]_{\overline{\alpha}} \ar[ru]_{\overline{f}} & &
 }
\end{equation*}
In $\mathsf{PreOrd}(\mathbb C)$, if we write $A''$ and $\overline{A''}$ for the complements of $A'$ and $\overline{A'}$ in $A$, respectively, we have the decompositions $$A = A' \coprod A'' = A_0 \coprod A_1 \coprod A''$$ and $$A = \overline{A'} \coprod \overline{A''} = A_0 \coprod \overline{A_1} \coprod \overline{A''}.$$
Accordingly, by taking into account the distributivity law for subobjects (see Lemma $1.4.2$ in \cite{Elep} and recall that coproducts in $\PreOrd C$ are computed ``componentwise'' \cite[Proposition~5.3]{BCG}) we get the following equalities: 
\begin{eqnarray}
 A  &=&  (A_0 \coprod A_1 \coprod A'' ) \cap (A_0 \coprod \overline{A_1} \coprod \overline{A''})\nonumber \\
  & = & (A_0 \cap A_0) \coprod (A_0 \cap \overline{A_1}) \coprod (A_0 \cap \overline{A''})\coprod (A_1 \cap A_0) \coprod (A_1 \cap \overline{A_1}) \coprod (A_1 \cap \nonumber \overline{A''}) \coprod \nonumber \\& &  ({A''} \cap A_0)
  \coprod ({A''} \cap \overline{A_1}) \coprod (A'' \cap \overline{A''}). \nonumber 
\end{eqnarray}
By taking into account the equalities $A_0 \cap A_0 = A_0$ and $$A_0 \cap \overline{A_1} = A_0 \cap \overline{A''} = A_1 \cap A_0 = A'' \cap A_0 =0,$$
we see that
$$A =  A_0 \coprod (A_1 \cap \overline{A_1}) \coprod (A_1 \cap \nonumber \overline{A''}) 
  \coprod ({A''} \cap \overline{A_1}) \coprod (A'' \cap \overline{A''}).$$
  We then observe that:
  \begin{itemize}
  \item $f=\overline{f}$ on $A_0$;
  \item $f$ is trivial on $A_1$ and $\overline{f}$ is trivial on $\overline{A_1}$, hence $f$ and $\overline{f}$ are trivial on $A_1 \cap \overline{A_1}$;
  \item $f$ is trivial on $A_1$ and $< \overline{\alpha}, \overline{f}>$ is zero in $\mathsf{Stab}(\mathbb C)$ on $\overline{A''}$, hence $< {\alpha}, {f} >$ and $< \overline{\alpha}, \overline{f}>$ are zero morphisms on $A_1 \cap \overline{A''}$ in $\mathsf{Stab}(\mathbb C)$;
  \item similarly, $< {\alpha}, {f} >$ and $< \overline{\alpha}, \overline{f}>$ are zero morphisms on $A'' \cap \overline{A_1}$;
  \item $< {\alpha}, {f} >$ is zero on $A''$ and $< \overline{\alpha}, \overline{f}>$ is zero on $\overline{A''}$, and this implies that $< {\alpha}, {f} >$ and $< \overline{\alpha}, \overline{f}>$ are both zero on $A'' \cap \overline{A''}$ in $\mathsf{Stab}(\mathbb C)$.
  \end{itemize}
  By assumption $G$ is a torsion theory functor, hence it sends the trivial morphisms in $\mathsf{PreOrd}(\mathbb C)$ to zero morphisms in $\mathbb X$. A zero morphism in $\mathsf{Stab}(\mathbb C)$ is a morphism of the form
  $$
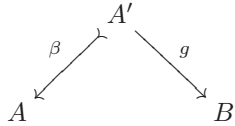
\xymatrix{& A'\ar@{ >->}[dl]_{\beta} \ar[dr]^{g} & \\ A & & B}
 $$
  where $g$ is trivial in $\mathsf{PreOrd}(\mathbb C)$. Accordingly, in this case, $\overline{G}(< \beta, g > )= [0,0]$ is the zero morphism from $A$ to $B$ in $\mathbb X$. By assumption $G$ preserves finite coproducts, hence 
  $$G(A) = G(A_0) \coprod G (A_1 \cap \overline{A_1}) \coprod G(A_1 \cap \nonumber \overline{A''}) 
  \coprod G({A''} \cap \overline{A_1}) \coprod G(A'' \cap \overline{A''}),$$
  and from the observations above we know that $[f,0]$ and $[\overline{f}, 0]$ coincide on $G(A_0)$ and are zero morphisms on all the other components. It follows that $[f,0] = [\overline{f}, 0]$, and the definition of $\overline{G}$ is compatible with the congruence defining the morphisms in $\mathsf{Stab}(\mathbb C)$. 
  
  To prove that $\overline{G}$ is a functor consider two composable morphisms in $\mathsf{Stab}(\mathbb C)$
  $$\xymatrix{A \ar@{.>}[r]^{< \alpha, f >} & B \ar@{.>}[r]^{< \beta, g >} & C, }$$
  and the following composition diagram where the upper square is a pullback
  
  $$\xymatrix@=30pt{ &\overline{A}^ c  \ar@{ >->}[rd]^{\overline{\alpha}^c}& & \overline{A}  \ar@{>->}[ld]_{\overline{\alpha}} \ar[dr]^{\overline{f}}  & & & \\
&&  A' \ar@{ >->}[dl]_{\alpha}  \ar[dr]^{f} &  & B' \ar@{ >->}[dl]_{\beta}  \ar[dr]^{g} & & \\
& A  \ar@{.>}[rr]_{\langle \alpha, f \rangle}   & & B \ar@{.>}[rr]_{\langle \beta, g \rangle}    &&C\\
 & & {A'}^c \ar@{>->}[ul]_{}  \ar[ur]_{0} &  & {B'}^c \ar@{ >->}[ul]_{}   \ar[ur]_{0}& & 
 } $$  
and ${A'}^c$, ${B'}^c$ are the complements of $A'$ and $B'$ in $A$ and in $B$, respectively. We also consider the pullback

 \begin{equation}\label{p-back} 
 \xymatrix{ \overline{A}^c  \ar[d] \ar@{>->}[r] & A' \ar[d]^f \\
{B'}^c   \ar@{>->}[r] & B
 }
 \end{equation}
 
 expressing the fact that $\overline{A}^c = f^{-1}({B'}^c)$, and we observe that 
 $$ A = A' \coprod  {A'}^c = \overline{A}^c  \coprod \overline{A} \coprod  {A'}^c.$$
 In $\mathbb X$ we have to check that $[G(g), 0] [G(f), 0]$ and $[G(g \overline{f}), 0]$ coincide on $$G(A) =  G(\overline{A}^c)  \coprod G(\overline{A}) \coprod  G({A'}^c).$$
 On $G(\overline{A})$ we have $G(g) G(\overline{f})$ in both cases, hence
 $$
 \overline{G} (< \beta, g >) \cdot \overline{G} (<\alpha, f >) = \overline{G}(< \beta, g > <\alpha, f >).
 $$
 Now, on $\overline{A}^c$ the morphism $f$ factors through ${B'}^c$ (see diagram \eqref{p-back}), hence
 $$
 \overline{G} (< \beta, g >) \cdot \overline{G}(<\alpha, f >)
 $$
 is the zero morphism on $G({\overline{A}}^c)$. But $\overline{G}(< \beta, g  > < \alpha, f >)$ is also the zero morphism on $G({\overline{A}}^c)$, hence these two morphisms are equal on $G({\overline{A}}^c)$. Finally, on ${A'}^c$ we have $< \alpha, f > = 0$, hence again $$\overline{G} (< \beta, g >) \cdot \overline{G} (<\alpha, f >) = 0 = \overline{G}(< \beta, g \rangle>< \alpha, f >) $$ on $G({A'}^c)$, completing this part of the proof.
 
One clearly has that $\overline{G}(1_A) = G(1_A) = 1_{G(A)}$, since $G$ is a functor. To see that $\overline{G}$ preserves finite coproducts one has to observe that $G$ preserves finite coproducts and these are calculated in $\mathsf{Stab}(\mathbb C)$ as in $\mathsf{PreOrd}(\mathbb C)$ (see Corollary $6.3$ in \cite{BCG}). 
In order to check that $\overline{G}$ is a torsion theory functor, since $\overline{G}$ and $G$ coincide on objects, it will suffice to prove that $\overline{G}$ preserves the canonical short exact sequences in the torsion theory. This follows from Proposition \ref{Sigma-is-tt}, since the canonical short exact sequence in the torsion theory in $\mathsf{Stab}(\mathbb C)$ is the image by $\Sigma$ of the canonical short $\mathcal Z$-exact sequence in the pretorsion theory in $\mathsf{PreOrd}(\mathbb C)$ and, by assumption, $G$ preserves this kind of sequences.
\end{proof}

\section{The case of $\tau$-pretoposes}
The aim of this section is to prove that, when $\mathbb C$ is a $\tau$-pretopos (in the sense of Definition \ref{tau}), all the short exact sequences in $\mathsf{Stab}(\mathbb C)$ are images (up to isomorphism) by the functor $\Sigma \colon \mathsf{PreOrd}(\mathbb C) \rightarrow \mathsf{Stab}(\mathbb C)$ of a short $\mathcal Z$-exact sequence in $\mathsf{PreOrd}(\mathbb C)$.

\begin{Prop}
The stable category $\mathsf{Stab}(\mathbb C)$ has disjoint binary coproducts.
\end{Prop}
\begin{proof}
Consider any commutative diagram in $\mathsf{Stab}(\mathbb C)$ 
$$
\xymatrix{ C \ar[r]^{< \alpha, f >} \ar[d]_{< \beta, g>} & A \ar[d]^{\Sigma(s_A) } \\
B \ar[r]_-{\Sigma(s_B)} & A \coprod B.
}
$$
where $s_A$ and $s_B$ denote the coprojections of the coproduct in $\PreOrd C$. This means that there is a congruence diagram
$$
\xymatrix@=35pt{
{A_0^{1}}^c \ar@{>->}[rr]^{} & & A_1 \ar@{>->}[ld]^{\alpha} \ar[r]^{f}  &A \ar[rd]^-{s_A} & & \\
A_0 \ar@{ >->}[drr]_{\alpha_1} \ar@{ >->}[rru]^{\alpha_0} \ar@{ >->}[r] & A   & & & A \coprod B   \\
{A_0^{2}}^c \ar@{>->}[rr]_{}  & & A_2 \ar@{>->}[ul]_{\beta} \ar[r]_{g} &B \ar[ru]_-{s_B}  & & 
 }
$$
In the category $\mathsf{PreOrd}(\mathbb C)$ the equality $s_A f \alpha_0 = s_B g \alpha_1$ induces a unique morphism $A_0 \rightarrow A \times_{A \coprod B} B$ to the pullback $A \times_{A \coprod B} B$ of $s_A$ and $s_B$. Since $A \times_{A \coprod B} B$ is the initial object $0$ in $\mathsf{PreOrd}(\mathbb C)$ and $0$ is strict, it follows that $A_0 = 0$. This implies that ${A_0^{1}}^c = A_1$ and $ {A_0^{2}}^c = A_2$, and the morphisms $s_A f$ and $s_B g$ are both trivial.  Since $\Sigma(s_A)$ and $\Sigma(s_B)$ are monomorphisms in $\mathsf{Stab}(\mathbb C)$ (by Proposition \ref{Sigma-is-tt}), it follows that $< \alpha, f >= 0$ and $< \beta, g> =0$ . From the fact that $0$ is a zero object in $\mathsf{Stab}(\mathbb C)$ it follows that the square 
$$
\xymatrix{ 0 \ar[r]^{} \ar[d]_{} & A \ar[d]^{\Sigma(s_A) } \\
B \ar[r]_-{\Sigma(s_B)} & A \coprod B.
}
$$
is a pullback in $\mathsf{Stab}(\mathbb C)$, as desired.

\end{proof}
\begin{definition}\label{def.preuniversal}
Let $\mathbb X$ be a category with binary coproducts. One says that binary coproducts in $\mathbb X$ are \emph{pre-universal} if, given any morphism 
$f \colon C \rightarrow A \coprod B$, there exists a commutative diagram of the form
\begin{equation}\label{pre-universal}
\xymatrix@=30pt{
 {A' } \ar[r]^{s_{A'}} \ar@{.>}[d]_{f_A}  \ar[d]  & C  \ar[d]_{f} &  {B'}  \ar[l]_{s_{B'}}  \ar[d]^{f_B}  \\
A \ar[r]_-{{s}_A} & A \coprod B & B \ar[l]^-{{s}_B} 
}
\end{equation}
where the top row of the diagram is a sum (i.e. $C = A' \coprod B'$ and $s'_A$ and $s'_B$ are the coprojections).
\end{definition}

Notice that the (non-extensive) category of pointed sets has pre-universal binary coproducts. We now want to prove that also the stable category $\mathsf{Stab}(\mathbb C)$ has this property. For ease of notation, since in $\mathsf{Stab}(\mathbb C)$ binary coproducts (exist and) are computed as in $\PreOrd C$ \cite[Corollary~6.3]{BCG}, in the sequel we shall often write $\xymatrix{A \ar[r]^-{{s}_A} & A \coprod B & B \ar[l]_-{{s}_B} }$ for the coprojections of the coproduct of $A$ and $B$ both in $\mathsf{Stab}(\mathbb C)$ and in $\PreOrd C$.

\begin{Prop}\label{Stable-preuniversal}
The stable category $\mathsf{Stab}(\mathbb C)$ has pre-universal binary coproducts.
\end{Prop}
\begin{proof}
Let us consider any morphism $< \alpha, f >\colon C \rightarrow A \coprod B$, and then the diagram
$$
\xymatrix@=30pt{C' \ar[rrd]^{\Sigma(f)} \ar@{>->}[rd]_{\Sigma(\alpha)}& & \\
& C \ar[r]^(.4){< \alpha, f >} & A \coprod B \\
{C'}^c \ar[rru]_0 \ar[ru]^{\Sigma(\gamma)} & & 
}
$$
 in $\mathsf{Stab}(\mathbb C)$ and 
$$
\xymatrix@=30pt{
A''  \ar[r]^{s_{A''}} \ar@{.>}[d]_{f_A''}  \ar[d]  & C'  \ar[d]_{f} &  B''  \ar[l]_{s_{B''}}  \ar[d]^{f_B''}  \\
A \ar[r]_-{{s}_A} & A \coprod B & B \ar[l]^-{{s}_B} 
}
$$
 in $\mathsf{PreOrd}(\mathbb C)$, respectively, 
where in the second one $A''$ and $B''$ are the inverse images along $f$, and then $C' = A'' \coprod B''$. Note that, by Proposition \ref{justification}, in the stable category we have the equality  $< \alpha, f> \Sigma(\gamma) = 0$.
There is then the following factorisation $\Sigma(f_A)$ in $\mathsf{Stab}(\mathbb C)$:
$$
\xymatrix@=40pt{A'' \ar[rrd]^{\Sigma(f_A'')} \ar[rd]_{\sigma_{A''}}& & \\
& A'' \coprod {C'}^c\ar[r]^{\Sigma(f_A)} & A  \\
{C'}^c \ar[rru]_0 \ar[ru]^{\sigma_{{C'}^c}} & & 
}
$$
where $\sigma_{A''}$ and $\sigma_{{C'}^c}$ are the coproduct coprojections.
One then sets $A' = {A''} \coprod {{C'}^c}$ and gets the diagram 
\begin{equation}\label{pre-universal-proof}
\xymatrix@=30pt{
 {A' } \ar[r]^{s_{A'}} \ar@{.>}[d]_{\Sigma(f_A)}  \ar[d]  & C  \ar[d]_{<\alpha, f>} &  {B''}  \ar[l]_{\Sigma(\alpha) s_{B''}}  \ar[d]^{\Sigma(f_B'')}  \\
A \ar[r]_-{{s}_A} & A \coprod B & {B,} \ar[l]^-{{s}_B} 
}
\end{equation}
whose commutativity can be checked as follows.
We have
$$C = C' \coprod {C'}^c = A'' \coprod B'' \coprod {C'}^c = A' \coprod B''.
$$
By Proposition \ref{justification}, we know that
$$
< \alpha, f> s_{A'} \sigma_{A''} = < \alpha, f> \Sigma(\alpha) s_{A''}= \Sigma(f) s_{A''}
$$
and 
$$
< \alpha, f> s_{A'} \sigma_{C'^c} = <\alpha , f> \Sigma(\gamma) = 0.
$$
Since we also have that
$$
s_A \Sigma(f_A)  \sigma_{A''} = s_A \Sigma(f_A'') = \Sigma(f) s_{A''}.
$$
and
$$
s_A \Sigma(f_A)  \sigma_{C'^c} = s_A 0 = 0,
$$ we conclude that $< \alpha, f> s_{A'} = s_A \Sigma(f_A)$, as desired. 
On the other hand, the following equalities show that the right-hand side of diagram \eqref{pre-universal-proof} commutes: 
$$
< \alpha, f > \Sigma(\alpha)  s_{B''} = \Sigma(f)  s_{B''} = s_B \Sigma({f_B''}).
$$
\end{proof}

\begin{remark}
Let us observe that the choice of the objects $A'$ and $B'$ in Definition \ref{def.preuniversal} is by no means unique. Indeed, in the proof of Proposition \ref{Stable-preuniversal}, we could as well have chosen $A'=A''$ and $B'=B''\coprod {C'}^c$ (with reference to diagram \eqref{pre-universal-proof}). So the pre-universality of binary coproducts could be rephrased as the existence of three objects $A'', B'', C''$, respectively mapped by $f$ in $A, B, 0$, and such that $C=A'' \coprod B'' \coprod C''$.
\end{remark}
\begin{lemma}\label{sum-of-kernels}
Let $\mathbb X$ be a category with a zero object and binary coproducts which are disjoint and pre-universal. Assume that $$\xymatrix{K \ar[r]^k& A \ar[r]^f & B} $$ and $$\xymatrix{N \ar[r]^n& C \ar[r]^g & D} $$ are composable morphisms such that $k = \ker (f)$ and $n = \ker(g)$. Then the morphism $k \coprod n \colon K \coprod N \rightarrow A \coprod C$ is the kernel of $f \coprod g \colon A \coprod C \rightarrow B \coprod D$.
\end{lemma}
\begin{proof}
First of all the composite $(f \coprod g) (k \coprod n)$ is clearly the zero morphism:
$$(f \coprod g) (k \coprod n) = fk \coprod gn = 0 \coprod 0 = 0.$$
Next consider any arrow $h \colon E \rightarrow A \coprod C$ such that $(f \coprod g) h = 0$. By the pre-universality of binary coproducts we can form the commutative diagram
$$
\xymatrix{
 E_1 \ar[r]^{s_1} \ar[d]_{h_1}  \ar[d]  & E  \ar[d]_h & E_2 \ar[l]_{s_2} \ar[d]^{h_2}  \\
A \ar[r]_-{s_A} & A \coprod C & C \ar[l]^-{s_C}
}
$$
where $E = E_1 \coprod E_2$. We have the equality
$$s_B f h_1 = (f \coprod g) s_A h_1 = (f \coprod g) h s_1 =0,$$
where $s_B \colon B \rightarrow B \coprod D$ is a monomorphism, hence there is a unique morphism $m_1$ such that $k m_1 = h_1$. Similarly, one can prove that there is a unique $m_2$ such that $n m_2 = h_2$. The universal property of the coproduct $E =  E_1 \coprod E_2$ gives a unique morphism $m = m_1 \coprod m_2 \colon E \rightarrow K \coprod N$ such that
$$(k \coprod n ) m s_1 =  (k \coprod n) s_K m_1 = s_A k m_1 = s_A h_1 = h s_1.$$
Symmetrically, one has that
$(k \coprod n ) m s_2  = h s_2$, yielding the equality $(k \coprod n ) m = h$. 
To check the uniqueness of the factorization consider then another morphism $r \colon E \rightarrow K \coprod N$ such that $(k \coprod n ) r = h$. Again by pre-universality we have a commutative diagram
$$
\xymatrix{
 \tilde{E}_1 \ar[r]^{\tilde{s}_1} \ar[d]_{r_1}  \ar[d]  & E  \ar[d]_r & \tilde{E}_2 \ar[l]_{\tilde{s}_2} \ar[d]^{r_2}  \\
K \ar[r]_-{s_K} & K \coprod N & N \ar[l]^-{s_N}
}
$$
where $E=\tilde{E}_1 \coprod \tilde{E}_2$.
Consider the following commutative diagram, where $E_1 = \hat{E}_{11} \coprod  \hat{E}_{12}$ and $E_2 = \hat{E}_{21} \coprod  \hat{E}_{22}$, 
$$
\xymatrix@=30pt{
 \hat{E}_{11} \ar[r]^{\hat{s}_{11}} \ar[d]_{s_{11}}  \ar[d]  & {E}_1  \ar[d]_{s_1} &  \hat{E}_{12}  \ar[l]_{\hat{s}_{12}}  \ar[d]^{s_{12}}  \\
\tilde{E}_1 \ar[r]_-{\tilde{s}_1} & E & \tilde{E}_2 \ar[l]^-{\tilde{s}_2} \\
 \hat{E}_{21} \ar[r]_{\hat{s}_{21}} \ar[u]^{s_{21}}  & E_2 \ar[u]^{s_2}  &  \hat{E}_{22} \ar[l]^{\hat{s}_{22}} \ar[u]_{s_{22}}
}
$$
that exists by the pre-universality of binary coproducts. By assumption the restrictions to $ \hat{E}_{11}$ of $m$ and $r$ are equal when composed with $k \coprod n$
$$(k \coprod n) r s_1 \hat{s}_{11} = (k \coprod n) m s_1 \hat{s}_{11}$$
and these composites both factor through $s_A k$
$$ (k \coprod n) m s_1 \hat{s}_{11} = h s_1  \hat{s}_{11}  = s_A h_1  \hat{s}_{11}  =  s_A k  m_1 \hat{s}_{11}$$
and $s_A k$ is a monomorphism (as a composite of monomorphisms). This means that $m$ and $r$ induce a unique morphism $ m_1 \hat{s}_{11} = r_1  \hat{s}_{11} \colon  \hat{E}_{11} \rightarrow K$. 
Similarly, $m$ and $r$ also induce a unique morphism $m_2 \hat{s}_{22} = r_2 \hat{s}_{22}  \colon \hat{E}_{22} \rightarrow N$.
On $\hat{E}_{12}$ we get the following equalities
$$
 s_A h_1 \hat{s}_{12} =  h s_1 \hat{s}_{12}
 = (k \coprod n) m  s_1 \hat{s}_{12}
 = (k \coprod n) r  s_1 \hat{s}_{12}
 $$
 $$
 = (k \coprod n) r \tilde{s}_2 s_{12} 
 = (k \coprod n) s_N r_2 s_{12} 
 = s_C n r_2 s_{12}
$$
showing that the induced morphisms $\hat{E}_{12} \rightarrow A$ and $\hat{E}_{12} \rightarrow C$ are the zero morphisms, 
since the coproduct $A \coprod C$ is disjoint. But $k$ and $n$ are both monomorphisms, hence both the morphisms $\hat{E}_{12} \rightarrow K$ and $\hat{E}_{12} \rightarrow N$
are zero as well.
Similarly, the morphisms $\hat{E}_{21} \rightarrow K$ and $\hat{E}_{21} \rightarrow N$ are also zero. By composing with $s_K \colon K \rightarrow K \coprod N$ and $s_N \colon N \rightarrow K \coprod N$ we obtain that the restrictions $\hat{E}_{11} \rightarrow K \coprod N$ of $m$ and $r$ are equal. Similarly, the restrictions $\hat{E}_{12} \rightarrow K \coprod N$ of $m$ and $r$ are both zero, hence the restriction of $m$ and $r$ to $E_1$ are equal. In a similar way one checks that the restrictions of $m$ and $r$ to $E_2$ are equal, and, finally, $m=r$.
\end{proof}
\begin{Prop}\label{kernels-in-Stab}
The category $\mathsf{Stab}(\mathbb C)$ has kernels. If $< \alpha, f >\colon A \rightarrow B$ is a morphism in $\mathsf{Stab}(\mathbb C)$ as in diagram \eqref{morphism-f}, its kernel is given by
$$\xymatrix@=35pt{{\Sigma (k \coprod 1_{{A'}^c}) \, \colon \quad } K \coprod {{A'}^c} \ar[r]^-{  } & A' \coprod {{A'}^c}} $$
where $k$ is the ${\mathcal Z}$-kernel of $f$ in $\mathsf{PreOrd}(\mathbb C)$ and ${A'}^c$ is the complement of $A'$ in $A$.
\end{Prop}
\begin{proof}
In $\mathsf{PreOrd}(\mathbb C)$ we have the $\mathcal Z$-kernels $$\xymatrix{K \ar[r]^k & A' \ar[r]^f & B }$$ and 
$$\xymatrix{{A'}^c \ar@{=}[r] & {A'}^c \ar[r]^{\tau} & 1. }$$
By Proposition $7.1$ in \cite{BCG} we know that $\Sigma \colon \mathsf{PreOrd}(\mathbb C) \rightarrow \mathsf{Stab}(\mathbb C)$ sends these $\mathcal Z$-kernels to the kernels 
$$\xymatrix{K \ar[r]^k & A' \ar[r]^f & B }$$ and $$\xymatrix{{A'}^c \ar@{=}[r] & {A'}^c \ar[r] & 0 }$$
in $\mathsf{Stab}(\mathbb C)$, respectively. By Lemma \ref{sum-of-kernels} we know that $k \coprod 1_{{A'}^c}$ is then the kernel of $f \coprod 0 \colon A' \coprod {A'}^c \rightarrow B \coprod 0 = B$. From Proposition \ref{justification} it follows that in the following diagram in $\mathsf{Stab}(\mathbb C)$
$$\xymatrix@=35pt{K \coprod {A'}^c \ar[r]^-{k \coprod 1_{{A'}^c}}& A \ar[r]^{< \alpha, f >} & B }$$
the morphism $k \coprod 1_{{A'}^c}$ is the kernel of $< \alpha, f >$, as desired.
\end{proof}
We now recall the following definition that had a role in proving some of the results in \cite{BCG}:
\begin{definition}\label{tau}
A $\tau$-pretopos is a pretopos $\mathbb C$ with the property that the transitive closure of any relation on an object in $\mathbb C$ exists in $\mathbb C$.
\end{definition}

Any $\sigma$-pretopos (i.e. a pretopos admitting denumerable unions of subobjects, that are preserved by pullbacks) is in particular a $\tau$-pretopos \cite{Elep}, as well as any elementary topos \cite[Proposition~7.7]{BCG}.

\begin{Prop}\label{Stab-has-kernels}
When $\mathbb C$ is a $\tau$-pretopos, two composable morphisms in $\mathsf{Stab}(\mathbb C)$ 
$$\xymatrix{A \ar@{.>}[r]^{< \alpha, f >} & B \ar@{.>}[r]^{< \beta, g >}&C }$$
form a short exact sequence if and only if, up to isomorphism, they are the image by $\Sigma$ of a short exact sequence in $\mathsf{PreOrd}(\mathbb C)$.
\end{Prop}
\begin{proof}
One implication follows from Theorem $7.14$ in \cite{BCG} (for which the assumption that $\mathbb C$ is a $\tau$-pretopos is needed).
Conversely, consider a short exact sequence 
\begin{equation}\label{exact-seq-stab1}
\xymatrix{A \ar@{.>}[r]^{< \alpha, f >} & B \ar@{.>}[r]^{<\beta, g >}&C }
\end{equation}
in $\mathsf{Stab}(\mathbb C)$. By Proposition \ref{kernels-in-Stab} we know that the kernel of $< \beta, g >$ is a morphism of type $\Sigma(n)$:
$$< \alpha, f > = \ker < \beta, g > = \Sigma (n).$$
More precisely, going back to the construction of the kernel $\mathsf{Ker}(< \beta, g>)$ as in Proposition \ref{kernels-in-Stab} we have the diagram
$$
\xymatrix@=45pt{ & K \ar@{>->}[r]^k  \ar@{>->}[dl]_{s_1} & B' \ar@{ >->}[dl]_{\beta} \ar[dr]^(.7)g  & {B'}^c \ar@{>->}[dll]^(.6){\beta^c} & \\
K \coprod {B'}^c \ar[r]_(.6){k \coprod 1_{{B'}^c}}& B & & C  & 
}
$$
and we set $n = k \coprod 1_{{B'}^c}$.
When $\mathbb C$ is a $\tau$-pretopos the functor $\Sigma \colon \mathsf{PrOrd}(\mathbb C) \rightarrow \mathsf{Stab}(\mathbb C)$ sends $\mathcal Z$-cokernels to cokernels (Corollary $7.13$ in \cite{BCG}): this implies that 
\begin{eqnarray}
< \beta , g > & = & \mathsf{coker} (< \alpha , f > ) \nonumber \\
& = &  \mathsf{coker} (\Sigma (n) ) \nonumber \\
& =& \Sigma ({\mathcal Z}{\rm{-}}\mathsf{coker} (n)). \nonumber
\end{eqnarray}
 The construction of the cokernel $\mathsf{coker}(< \alpha, k> )$ as in Lemma $7.11$ of \cite{BCG} shows that 
the $\mathcal Z$-cokernel $q \colon B \rightarrow Q$ of $n$ is such that 
$$\Sigma (q) = \mathsf{coker}(< 1_{K \coprod {B'}^c}, k \coprod 1_{{B'}^c}>) = \mathsf{coker}(\Sigma (n)).$$ Since the sequence \eqref{exact-seq-stab1} is exact, $< \beta, g >$ is isomorphic to $\Sigma (q)$, and the sequence \eqref{exact-seq-stab1} is isomorphic to the exact sequence
$$
\xymatrix{\Sigma (K \coprod {B'}^c) \ar@{.>}[r]^-{\Sigma (n)} & \Sigma(B) \ar@{.>}[r]^-{\Sigma (q)}& \Sigma (Q), }
$$
as desired.
\end{proof}
\begin{corollary}
When $\mathbb C$ is a $\tau$-pretopos, under the assumptions of Theorem \ref{universal}, the functor $\overline{G}$ preserves the short exact sequences whenever the functor $G$ sends short $\mathcal Z$-exact sequences to short exact sequences.
\end{corollary}
\begin{proof}
This follows immediately from Proposition \ref{Stab-has-kernels}.
\end{proof}
\begin{corollary}
When $\mathbb C$ is a $\tau$-pretopos, under the assumptions of Theorem \ref{universal}, the functor $\overline{G}$ preserves cokernels whenever the functor $G$ sends $\mathcal Z$-cokernels to cokernels.
\end{corollary}
\begin{proof}
Let $< \alpha, f > \colon A \rightarrow B$ be a morphism in $\mathsf{Stab}(\mathbb C)$ and $q \colon B \rightarrow C$ its cokernel. As we have noticed in the proof of Proposition \ref{Stab-has-kernels}, this cokernel is the image by $\Sigma$ of the $\mathcal Z$-cokernel of $f$, i.e. $q = \Sigma ({\mathcal Z}$-$\mathsf{coker}(f))$. In the category $\mathbb X$ we have that $\overline{G} (q) = G(q) = \mathsf{coker}(G(f)) = \mathsf{coker}(\overline{G}(f)).$
Consider then the following diagram 
$$
\xymatrix@=35pt{
 & & G(A')  \ar[ld]_{G(\alpha)} \ar[dr]^{G(f)}  & & X  &  \\
  & G(A) = G(A') \coprod G({A'}^c)   \ar[rr]^-{[G(f), 0]} & &G(B) \ar[r]^{G(q)} \ar[ru]^x & G(Q)  \ar@{.>}[u]_{\exists ! y} \\
   & & G({A'}^c) \ar[ul]^{G(\alpha^c)} \ar[ru]_{0} & & & 
 }
$$
One sees that $G(q) G(f) =0$, since $G(q) = \mathsf{coker}(G(f))$ and, obviously, $G(q) 0= 0$, hence $G(q) [G(f), 0] = 0$. Then, if $x[G(f),0]=0$ we get 
$$x G(f) = x [G(f),0] G(\alpha) = 0 G(\alpha)=0,$$
yielding a unique factorisation $y$ of $x$ through $G(q) = \mathsf{coker}(G(f))$. It follows that $G(q) = \mathsf{coker}([G(f),0])$, and this implies that
$$\overline{G}(q) = G(q) = \mathsf{coker}([G(f), 0]) = \mathsf{coker} (\overline{G}(< \alpha, f >)).$$
\end{proof}

\begin{Prop}
If in Theorem \ref{universal} we also assume that 
\begin{itemize}
\item $\mathbb C$ is a $\tau$-pretopos,
\item $\mathbb X$ has finite coproducts that are disjoint and pre-universal,
\item $G \colon \mathsf{PreOrd}(\mathbb C) \rightarrow \mathbb X$ sends $\mathcal Z$-kernels to kernels,
\end{itemize}
then the functor $\overline{G} \colon \mathsf{Stab}(\mathbb C) \rightarrow \mathbb X$ preserves kernels.
\end{Prop}
\begin{proof}
Consider a morphism $< \alpha, f > \colon A \rightarrow B$ in $\mathsf{Stab}(\mathbb C)$ and its kernel $k \coprod 1_{{A'}^c}$ in $\mathsf{Stab}(\mathbb C)$ as in Proposition \ref{kernels-in-Stab}
$$
\xymatrix@=35pt{
K \ar[rr]^k \ar[d]_{s_1} && A' \ar@{>->}[dl]_{\alpha} \ar[dr]^(.7)f  & &{A'}^c \ar[lllld]_(.8){s_2} \ar[dl]^{\alpha'} \\
K \coprod {A'}^c \ar[r]_-{k \coprod 1_{{A'}^c}}&A' \coprod {A'}^c \ar@{.>}[rr]_{< \alpha , f > } & &B, }
$$
where $k$ is the $\mathcal Z$-kernel of $f$. From the assumptions it follows that
$$\overline{G} (k) = G(k) = \mathsf{ker}(G(f)) = \mathsf{ker}(\overline{G}(f)),$$
and one also has 
$$1_{\overline{G}({A'}^c)} = \mathsf{ker} (0), \qquad \mathrm{where} \, \, 0 \colon \overline{G}({A'}^c) \rightarrow 0.$$
From Lemma \ref{sum-of-kernels} and Theorem \ref{universal} we get the equalities
$$\overline{G} (k \coprod 1_{{A'}^c}) = \overline{G}(k) \coprod 1_{\overline{G}({A'}^c)} = \mathsf{ker} (\overline{G}(f) \coprod 0) = \mathsf{ker}(\overline{G} (< \alpha, f> )),$$
hence $\overline{G}$ preserves the kernel of $< \alpha, f >$.
\end{proof}

Thus, in case of a $\tau$-pretopos, we have the following result.

\begin{theorem}\label{universal-tau}
Let $\mathbb C$ be a $\tau$-pretopos. The functor $\Sigma \colon \mathsf{PreOrd}(\mathbb C) \rightarrow \mathsf{Stab}(\mathbb C)$ is universal among all finite coproduct preserving \emph{torsion theory functors} $G \colon \mathsf{PreOrd}(\mathbb C)$ $\rightarrow \mathbb X$, where $\mathbb X$ has a torsion theory $({\mathcal T}, {\mathcal F})$, and it has binary coproducts that are disjoint and pre-universal. Consider any finite coproduct preserving torsion theory functor $G \colon \mathsf{PreOrd}(\mathbb C) \rightarrow \mathbb X$ that sends $\mathcal Z$-kernels and $\mathcal Z$-cokernels (then in particular short $\mathcal Z$-exact sequences) to kernels, cokernels (and short exact sequences). The functor $G$ then factors uniquely through $\Sigma$
$$ \xymatrix{\mathsf{PreOrd}(\mathbb C)  \ar[rr]^{\Sigma} \ar[rd]_{\forall G} & & \mathsf{Stab}(\mathbb C) \ar@{.>}[ld]^{\exists ! \overline{G}} \\
& \mathbb X. &
}
$$
and the induced functor $\overline{G}$ preserves finite coproducts, is a torsion theory functor that preserves kernels and cokernels (then in particular short exact sequences).
\end{theorem}

\end{document}